\documentclass[11pt]{article}
\usepackage[T1]{fontenc}
\usepackage[utf8]{inputenc}
\usepackage{authblk}

\usepackage{amsmath}
\usepackage{amssymb}
\usepackage{amsthm}
\usepackage{epsfig}
\usepackage{url}
\usepackage{hyperref}
\usepackage{array}
\usepackage{tikz}
\usetikzlibrary{patterns}
\usepackage{textcomp}
\usepackage{float}
\usepackage{color}

\newtheorem{theorem}{Theorem}

\newtheorem{corollary}[theorem]{Corollary}
\newtheorem{conjecture}[theorem]{Conjecture}
\newtheorem{lemma}[theorem]{Lemma}

\newtheorem{definition}[theorem]{Definition}

\newcommand{\ecc}[1]{\mbox{ecc}#1}
\newcommand{\rad}[1]{\mbox{rad}#1}

\newcommand{\G}{\mathcal{G}}
\newcommand{\TF}{\mathcal{TF}}

\title{Burning Graph Classes}
 \author[1]{Mohamed Omar\thanks{\textcolor{blue}{\href{mailto:omar@g.hmc.edu}{omar@g.hmc.edu}} This research was supported by the AMS Claytor-Gilmer Fellowship, the Karen EDGE Fellowship, and the Harvey Mudd College Faculty Research, Scholarship, and Creative Works Award.}}
 \author[2]{Vibha Rohilla\thanks{\textcolor{blue}{\href{mailto:vrohilla@g.hmc.edu}{vrohilla@g.hmc.edu}} This research was supported by the Harvey Mudd College Department of Mathematics Giovanni Borrelli Fellowship.}}
 \affil[1,2]{{Department of Mathematics, Harvey Mudd College}}

\begin{document}
\maketitle

\begin{abstract}
The Burning Number Conjecture, that a graph on $n$ vertices can be burned in at most $\lceil \sqrt{n} \ \rceil$ rounds, has been of central interest for the past several years. Much of the literature toward its resolution focuses on two directions: tightening a general upper bound for the burning number, and proving the conjecture for specific graph classes. In the latter, most of the developments work within a specific graph class and exploit the intricacies particular to it. In this article, we broaden this approach by developing systematic machinery that can be used as test beds for asserting that graph classes satisfy the conjecture. We show how to use these to resolve the conjecture for several classes of graphs including triangle-free graphs with degree lower bounds, graphs with certain linear lower bounds on $r$-neighborhood sizes, all trees whose non-leaf vertices have degree at least $4$, trees whose non-leaf vertices have degree at least $3$ (on at least $81$ vertices), trees whose non-leaf vertices are less than $\frac{2}{3}$ concentrated in degree $2$, and trees with a low concentration of high degree non-leaf vertices (the last two results holding for sufficiently many non-leaf vertices).
\end{abstract}

\section{Introduction}\label{sec:intro}

Recently, a mathematical model for analyzing contagion spread was introduced by Bonato, Janssen and Roshanbin \cite{bonato2014burning},\cite{bonato2016burn}. In this model, contagion is treated as a fire spreading through a network. Given an undirected simple graph modeling the network at hand, fire spreads through the graph in rounds in the following manner. In the first round, a vertex is chosen to be lit by fire. In each subsequent round, two events occur. Firstly, all vertices that are on fire spread the fire to vertices adjacent to them. Secondly, a new vertex in the network is selected to be lit by fire. This process continues until all vertices are on fire. 

The central concern in this model is how quickly a contagion can spread in a given network. A detailed understanding of the mechanisms that influence the speed of spread is motivated by practical applications, such as understanding the propagation of fake news or political propaganda through social networks.  For a given network this is inherently tied to the sequence of vertices that are selected to be newly lit in each round. This motivates the central statistic of interest in this model, the \emph{burning number} of a graph $G$. Denoted by $b(G)$, the burning number is the minimum number of rounds needed to complete the burning of $G$.

There are many graph classes for which the burning number is explicitly or asymptotically known. Letting $P_n$ be a path on $n$ vertices, it was shown in \cite{bonato2016burn} that $b(P_n)=\lceil \sqrt{n} \ \rceil$. For grid graphs $G_{m,n}$ that can be recognized as the graph product $P_m \square P_n$, it was proven in \cite{mitsche2017burning} that $b(G_{m,n})$ is asymptotically either $\Theta(\sqrt{n})$ or $O \left( \sqrt[3]{\frac{3}{2} mn} \right)$, depending on the relationship between $m$ and $n$. Similar asymptotic results hold for the strong product $P_m \boxtimes P_n$ \cite{bonato2016burn}. The $n$-cube was proven to have a burning number that is asymptotically logarithmic in the number of vertices \cite{mitsche2018burning}. A unifying phenomenon in all these graph classes and many others in the literature is that the burning number is at most $\lceil \sqrt{n} \ \rceil$ for any graph on $n$ vertices. The  central unanswered question on graph burning is whether this is true for all connected graphs. 

\begin{conjecture}[Burning Number Conjecture, \cite{bonato2014burning}]\label{conjecture:main}
Let $G$ be a connected graph on $n$ vertices. Then $b(G) \leq \lceil \sqrt{n} \ \rceil.$
\end{conjecture}

A graph that satisfies the Burning Number Conjecture is said to be \emph{well-burnable}, and though the conjecture has not been resolved, the state of the art bound of $\sqrt{\frac{4n}{3}} + O(1)$ was recently established by Bonato et al \cite{bonato2021improved} and independently Bastide et al \cite{improvedburning}. Alongside improving the general upper bound for $b(G)$, a common theme amongst papers in the graph burning literature is proving that specific graph classes are well-burnable. For instance, in addition to the previously mentioned graph classes, classes of trees including spiders and caterpillars have been proven to be well-burnable \cite{bonato2019bounds},\cite{das2018burning},\cite{liu2020burning}. The contribution of this paper to the literature is to complement these findings but in a more general framework. Instead of establishing that specific graph classes are well-burnable, we present general paradigms that can be used to do so. These constructs, developed in Theorem~\ref{thm:main} and Theorem~\ref{thm:trees}, allow us to prove many graph classes are well-burnable. The classes include triangle-free graphs with degree lower bounds (Theorem~\ref{thm:trianglefree}), graphs whose closed $r$-neighborhoods have sizes at least linear in $r$ (for $r \geq 2$, as seen in Theorem~\ref{thm:main2}), trees whose non-leaf vertices have degree at least $4$, trees on at least $81$ vertices whose non-leaf vertices have degree at least $3$ (both in Theorem~\ref{thm:trees2}), and trees whose non-leaf vertices are not too concentrated in degree $2$ (Theorem~\ref{thm:prob}) or are barely concentrated in high degree (Theorem~\ref{thm:prob2}), the last two classes on sufficiently many non-leaf vertices. These widespread applications suggest that our developments are fruitful for establishing many other graph classes as well-burnable.

After preliminaries in Section~\ref{sec:prelim}, we continue to our two specific paradigms. Section~\ref{sec:families} focuses on the first of these, which provides conditions for establishing that a graded family of graphs is eventually well-burnable if its $r$-neighborhoods satisfy a lower bound condition. Section~\ref{sec:trees} focuses on the second of these, which establishes a condition on the degree distribution of trees that guarantees a tree is well-burnable. We conclude with future directions in Section 5.

\section{Preliminaries}\label{sec:prelim}

We begin our discussion with necessary graph theoretic terminology. Throughout our discussion, our graphs will be finite, simple, undirected and connected.

The first set of statistics we introduce focus on distance metrics in a graph. Given a graph $G$ and vertices $u,v$ in $G$, the \emph{distance} between $u$ and $v$, denoted $d(u,v)$, is the number of edges in the shortest path from $u$ to $v$. For a fixed vertex $v$, the maximum distance from $v$ to any other vertex in $G$ is said to be its \emph{eccentricity} and is denoted $\ecc(v)$. The \emph{radius} of $G$, denoted $\rad(G)$, is the minimum eccentricity achieved by any of its vertices. That is, $\rad(G) = \min_{v \in V(G)} \ecc(v).$ Loosely speaking, a vertex $v$ whose eccentricity is $\rad(G)$ is centrally located in the graph. 


A closely related concept that our results in Section~\ref{sec:families} heavily hinge on is the closed $r$-neighborhood of a vertex in a graph. For a vertex $v$, and a fixed positive integer $r$, the \emph{closed $r$-neighborhood} of $v$, denoted $N_r(v)$, is the set of vertices of distance at most $r$ from $v$. That is $N_r(v) = \{ u \in V(G) \ : \ d(u,v) \leq r \}.$ For a fixed positive integer $r$, closed $r$-neighborhoods  of vertices inform an upper bound on the burning number through the following lemma, which encapsulates ideas from Section 2 of \cite{bonato2019approximation} but whose proof is in Section 2 of \cite{kamali2020burning}. We include a proof for completeness.
\begin{lemma}[\cite{kamali2020burning}]\label{lem:maximal}
For a graph $G$ and a fixed positive integer $r$, suppose $A$ is a maximal set of vertices in $G$ with respect to having disjoint closed $r$-neighborhoods. Then $b(G) \leq |A| + 2r.$
\end{lemma}
\begin{proof}
It suffices to prove that we can complete a burning process for $G$ in at most $|A|+2r$ rounds. First, observe that any vertex $v$ in $G$ that is not in $A$ is at distance at most $2r$ from some vertex in $A$, for otherwise, the vertices in $A \cup \{v\}$ would have disjoint $r$-neighborhoods, contradicting the maximality of $A$. Now, to burn $G$, in the first $|A|$ rounds, burn the vertices in $A$, one in each round (if we run out of vertices in $k<|A|$ rounds, burn a random vertex in each of the remaining at most $|A|-k$ rounds). Since every vertex in $G$ outside of $A$ is at distance at most $2r$ from some vertex in $A$, the burning process will end in at most $|A|+2r$ rounds.
\end{proof}

Another result in the literature that will be useful to us is the Tree Reduction Theorem \cite{bonato2016burn}. This tells us that if we want to know the burning number of a graph $G$, it suffices to look at the burning number of its spanning trees.

\begin{theorem}[Tree Reduction Theorem, \cite{bonato2016burn}]\label{thm:treereduction}
For a graph $G$, we have that
\[
b(G) = \min \{b(T) \ : \ T \mbox{ is a spanning tree of } G\}.
\]
\end{theorem}
This in particular gives us the following useful consequence which we will exploit.

\begin{corollary}\label{cor:caterpillar}
Let $G$ be a graph on $n$ vertices in which every vertex has degree at least $d>\frac{n-1}{3}-1$. Then $G$ is well-burnable.
\end{corollary}
\begin{proof}
The condition that $d>\frac{n-1}{3}-1$ implies by \cite{broersma1988existence} that $G$ has a spanning caterpillar, say $T$. By \cite{liu2020burning}, $T$ is well-burnable. However by Theorem~\ref{thm:treereduction}, $b(G) \leq b(T)$ and so $G$ is well-burnable.
\end{proof}

Finally, we recall what is known about the burning number of graphs that have degree lower bounds. A breakthrough result in this direction was proven in \cite{kamali2020burning}.
\begin{theorem}[\cite{kamali2020burning}]\label{thm:23}
If $G$ is a connected graph in which every vertex has degree at least $23$ then $G$ is well-burnable.
\end{theorem}
This result was improved significantly for connected graphs on sufficiently many vertices in a recent article \cite{improvedburning}.
\begin{theorem}[\cite{improvedburning}]\label{thm:improved1}
If $G$ is a connected graph on sufficiently many vertices, and every vertex in $G$ has degree at least $4$, then $G$ is well-burnable.
\end{theorem}
In that same article, connected graphs whose vertices have degree at least $3$ are almost proven to be well-burnable.
\begin{theorem}[\cite{improvedburning}]\label{thm:improved2}
If $G$ is a connected graph on $n$ vertices in which every vertex has degree at least $3$, then $b(G) \leq \lceil \sqrt{n} \ \rceil + 2.$
\end{theorem}

\section{Graded Families of Graphs}\label{sec:families}

A \emph{graded family of connected graphs} $\G$ is a class of connected graphs stratified by the positive integers, written as a union $\displaystyle \G=\cup_{d=1}^{\infty} \G_d$. In this section we develop machinery that takes in a graded family $\mathcal{G}$ and generates necessary conditions for $\G_d$ to be well-burnable for sufficiently large $d$. For our first main theorem, the key insight is inspired by our observation from Lemma~\ref{lem:maximal}: if one can get a universal lower bound on the number of vertices in any closed $r$-neighborhood of any vertex in a graph, then one can get an upper bound on the burning number. This motivates our next definition.

\begin{definition}
Let $\G$ be a graded family of connected graphs. We say that $\G$ is \emph{tethered} if there exist functions $f_1(x),f_2(x),\dots$ defined on $[1,\infty)$ so that for any positive integer $d$, any $G \in \G_d$, any $r \in \{1,2,\ldots,\rad(G)\}$, and any $v \in V(G)$, we have $|N_r(v)| \geq f_d(r).$ We call the functions $\{f_d(x)\}_{d \geq 1}$ a \emph{tethering} for $\G$.
\end{definition}
With the concept of a tethering established, we can now present our first main theorem.
\begin{theorem}\label{thm:main}
Let $\G$ be a graded family of connected graphs. Furthermore, suppose $\G$ is tethered with tethering $\{f_d(x)\}_{d \geq 1}$. For fixed positive integers $n,d$, consider the function $g_{n,d}(x)$ defined on $[1,\infty)$  by
\[
g_{n,d}(x) = \frac{n}{f_d(x)} + 2x
\]
and suppose $g_{n,d}(x)$ achieves a minimum at $m_{n,d} \in [1,\infty)$. Let $M_{n,d}$ be any upper bound on the two values
\[
\lceil m_{n,d} \rceil, \hspace{0.1in} \min \left \{g_{n,d} \left( \lfloor m_{n,d} \rfloor \right), g_{n,d} \left( \lceil m_{n,d} \rceil \right)\right \}.
\]
Then for any graph $G \in \G_d$ on $n$ vertices, $b(G) \leq M_{n,d}.$
\end{theorem}
At first glance, Theorem~\ref{thm:main} seems very general and opaque. To illustrate its effectiveness, we apply it to a particular graph class. Let $\mathcal{TF}_d$ be the set of connected triangle-free graphs (i.e. graphs with no $3$-cycles) whose vertices have degree at least $d$. We use a combination of Theorem~\ref{thm:main} and a strategic tethering to establish the following:
\begin{theorem}\label{thm:trianglefree}
If $d \geq 12$ then any graph $G \in \mathcal{TF}_d$ is well-burnable.
\end{theorem}

A few key observations should be made about Theorem~\ref{thm:trianglefree} in relation to the graph burning literature. Firstly, we notice that Theorem~\ref{thm:trianglefree} improves Theorem~\ref{thm:23} for a specific graph class. Secondly, Theorem~\ref{thm:trianglefree} came to light at the same time as Theorem~\ref{thm:improved2}, however Theorem~\ref{thm:trianglefree} has no stipulation on sufficient vertex size whatsoever. Thirdly, the true motivation for Theorem~\ref{thm:trianglefree} is to demonstrate the power of the overall paradigm afforded by Theorem~\ref{thm:main} and its potential to be used in many different settings in the future.

\begin{proof}
We first find a tethering for $\TF = \bigcup_{d \geq 1} \TF_d$, the class of connected triangle-free graphs. Let $G \in \TF_d$ and $v \in V(G)$. Since $v$ has degree at least $d$, $|N_1(v)| \geq d+1$. Now for any $i$ let $D_i$ be the set of vertices of distance exactly $i$ from $v$. Then $N_2(v) = D_0 \cup D_1 \cup D_2$. Furthermore, since $D_1$ has at least $d$ vertices, every vertex in $D_1$ has degree at least $d$, and no two vertices in $D_2$ can be adjacent to the same vertex in $D_1$ (otherwise $G$ would have a $3$-cycle), we deduce $|D_2| \geq d(d-1)$. From this, 
\begin{equation}\label{eqn:neighborhood1}
|N_2(v)|=|D_0 \cup D_1 \cup D_2| \geq d^2+1 =  \left \lfloor \frac{2+3}{5} \right \rfloor  (d^2+1).
\end{equation}
Furthermore, for $r \in \{3,4,5,6\}$ we have 
\begin{equation}\label{eqn:neighborhood2}
|N_r(v)| \geq |N_2(v)| \geq d^2+1 =  \left \lfloor \frac{r+3}{5} \right \rfloor  (d^2+1).
\end{equation}
Now suppose $r \geq 7$. For $i \in \{1,2,\ldots,\lfloor \frac{r+3}{5} \rfloor -1\}$ define $D'_i =  \bigcup_{j=5i-2}^{5i+2} D_j.$ The vertices in $D'_i$ are all within distance $r$ from $v$ because $5i+2 \leq 5 \left( \lfloor \frac{r+3}{5} \rfloor -1\right)+2 \leq (r+3)-5+2=r$. Select a vertex $v_i \in D_{5i}$. Since $v_i$ has degree at least $d$ and each of its neighbors have degree at least $d$, $v_i$ together with its neighbors and their neighbors comprise at least $1+d+d(d-1)=d^2+1$ vertices, and all these vertices lie in $D_i'$. Subsequently, letting $D_0' = D_0 \cup D_1 \cup D_2$, $M = \left \lfloor \frac{r+3}{5} \right \rfloor -1 $, and using the fact that $D_0',D_1',\ldots,D_M'$ are pairwise disjoint, we have
\begin{equation}\label{eqn:neighborhood3}
|N_r(v)| \geq \sum_{i=0}^M |D_i'| \geq (M+1) \cdot (d^2+1) = \left \lfloor \frac{r+3}{5} \right \rfloor (d^2+1).
\end{equation}
From our lower bound for $|N_1(v)|$, equation~(\ref{eqn:neighborhood1}), equation~(\ref{eqn:neighborhood2}) and equation~(\ref{eqn:neighborhood3}), we get a tethering $\{f_d(x)\}_{d \geq 1}$ for $\TF$ given by
\[
f_d(x) = \begin{cases}
d+1 \hspace{0.5in} &\mbox{ if } 1 \leq x < 2 \\
\left \lfloor \frac{x+3}{5} \right \rfloor (d^2+1) \hspace{0.5in} &\mbox{ if } 2 \leq x.
\end{cases}
\]
Applying Theorem~\ref{thm:main}, we can construct $g_{n,d}$ as
\[
g_{n,d}(x) = \begin{cases}
\frac{n}{d+1}+2 \hspace{0.5in} &\mbox{ if } 1 \leq x < 2 \\
\frac{n}{\left \lfloor \frac{x+3}{5} \right \rfloor (d^2+1)} + 2x \hspace{0.5in} &\mbox{ if } 2 \leq x.
\end{cases}
\]
and we search for $m_{n,d}$. Notice $g_{n,d}(x)$ is constant on $[1,2)$, and for such $x$ the difference $g_{n,d}(x)-g_{n,d}(2)$ is $\frac{n}{d+1} - \frac{n}{d^2+1}-2$. This difference is nonnegative precisely when $n \geq 2d+4 + \frac{6d+2}{d^2-d}$. By Corollary~\ref{cor:caterpillar} we can assume $n \geq 3d+4$, so $n \geq 2d+4 + \frac{6d+2}{d^2-d}$ since $d \geq \frac{6d+2}{d^2-d}$ for $d \geq 12$. These arguments tell us that $m_{n,d}$ is the minimizer of \[\frac{n}{\left \lfloor \frac{x+3}{5} \right \rfloor (d^2+1)} + 2x\]
in the interval $[2,\infty)$. Replacing $x$ by $5x-3$ and observing that $\lfloor x \rfloor$ is constant in $[k,k+1)$ for a positive integer $k$, whereas $5x-3$ increases on $[k,k+1)$, we can select the minimizer $m_{n,d}$ to be the minimizer of the following function over $[1,\infty)$
\[
s(x) = \frac{n}{x (d^2+1)} + 2(5x-3).
\]
This function is convex in $x$ and so its minimizer occurs when $\frac{d}{dx} s(x)=0$, which sets $m_{n,d} = \sqrt{\frac{n}{10(d^2+1)}}.$ We can make the assumption that $m_{n,d} \geq 1$. This is because if $m_{n,d}<1$ then $b(G) \leq g_{n,d}(2) = \frac{n}{d^2+1}+4<14$ so $b(G) \leq 13$, whereas $n \geq d^2+1 > 12^2$ so $\lceil \sqrt{n} \rceil \geq 13$. Now 
\begin{align*}
g_{n,d}(\lceil m_{n,d} \rceil) &= \frac{n}{\left \lceil \sqrt{\frac{n}{10(d^2+1)}} \right \rceil \cdot (d^2+1)} + 2 \left( 5 \left \lceil \sqrt{\frac{n}{10(d^2+1)}} \ \right \rceil - 3 \right) \\
&\leq 2\sqrt{10} \cdot \sqrt{\frac{n}{d^2+1}} + 4.
\end{align*}
By Theorem~\ref{thm:main}, $b(G) \leq 2\sqrt{10} \cdot \sqrt{\frac{n}{d^2+1}} +4 $ so $G$ is well-burnable if 
\[
\sqrt{n} \geq \frac{4}{1 - \sqrt{\frac{40}{d^2+1}}}.
\]
By Corollary~\ref{cor:caterpillar}, this forces every graph in $\mathcal{TF}_d$ to be well-burnable provided that $\sqrt{3d+4}$ exceeds the lower bound on the right, which occurs when $d \geq 11$. So altogether, if $d \geq 12$ any graph in $\mathcal{TF}_d$ is well-burnable.
\end{proof}

How can we use Theorem~\ref{thm:main} in more general settings? Theorem~\ref{thm:23} and Theorem~\ref{thm:improved1} are like Theorem~\ref{thm:trianglefree} in that they require degree lower bounds on the input graphs. However, there are many settings where there is no universal lower bound on degrees, for instance because of the existence of degree $1$ vertices. There may still however be lower bounds for the number of vertices in $r$-neighborhoods for $r \geq 2$. We could expect this for instance in a graph with high average degree but a few outlier vertices with very low degree. We can specialize Theorem~\ref{thm:main} to account for this by looking at graph classes that have a tethering $\{f_d(x)\}_{d \geq 1}$ where each function is linear in $x$ for $x \geq 2$.

\begin{theorem}\label{thm:main2}
Let $\G$ be a graded family of connected graphs with tethering $\{f_d(x)\}_{d \geq 1}$. Let $h(\cdot)$ be a function that is positive on $[1,\infty)$. Suppose that for every positive integer $d$,
\[f_d(x)=
\begin{cases}
1 \hspace{0.5in} &\mbox{ if } 1 \leq x < 2 \\
h(d) \cdot x \hspace{0.5in} &\mbox{ if } 2 \leq x.
\end{cases}\] 
If $n$ and $d$ are positive integers with
\begin{itemize}
\item $h(d)>8$
\item $n \geq 8 \cdot h(d)$
\item $\sqrt{n} \geq \frac{2}{1 - \sqrt{\frac{8}{h(d)}}}$
\end{itemize}
then any graph $G \in \G_d$ on $n$ vertices is well-burnable.
\end{theorem}

\begin{proof}
Using the same notation as Theorem~\ref{thm:main}, $g_{n,d}(x)=n+2x$ for $x \in [1,2)$ whereas $g_{n,d}(x) = \frac{n}{f_d(x)}+2x =  \frac{n}{h(d) \cdot x} + 2x$ on $[2,\infty)$. On this latter interval the function is convex in $x$ so a minimum will occur when $\frac{d}{dx} g_{n,d}(x) = 0$. This occurs when $m=\sqrt{\frac{n}{2 \cdot h(d)}}$ which is guaranteed to be in $[2,\infty)$ since $n \geq 8 \cdot h(d)$.  In fact $m$ is the minimizer of $g_{n,d}(x)$ on $[1,\infty)$ because for any $x \in [1,2)$, $\sqrt{\frac{n}{2 \cdot h(d)}} < n + 2 \leq n+2x=g_{n,d}(x)$. So we can set $m_{n,d}=m$. Now observe
    \begin{align*}
    g_{n,d} \left( \lceil m_{n,d} \rceil \right) &= \frac{n}{h(d) \cdot \left( \left \lceil \sqrt{\frac{n}{2 \cdot h(d)}} \ \right  \rceil \right)} + 2 \left( \left \lceil \sqrt{\frac{n}{2 \cdot h(d)}} \ \right  \rceil \right) \\
    & \leq \frac{n}{h(d) \cdot \sqrt{\frac{n}{2 \cdot h(d)}} } + 2 \left( \sqrt{\frac{n}{2 \cdot h(d)}} +1 \right) \\
    & = \sqrt{\frac{8n}{h(d)}} + 2,
    \end{align*}
   So $ \sqrt{\frac{8n}{h(d)}} + 2$ is an upper bound on $\min \left \{g_{n,d} \left( \lfloor m_{n,d} \rfloor \right), g_{n,d} \left( \lceil m_{n,d} \rceil \right)\right \}$. Furthermore, 
    \[
    \lceil m_{n,d} \rceil = \left \lceil \sqrt{\frac{n}{2 \cdot h(d)}} \right \rceil \leq \sqrt{\frac{n}{2 \cdot h(d)}}  + 1\leq \sqrt{\frac{8n}{h(d)}} + 2,
    \]
    so using Theorem~\ref{thm:main} we can set $M_{n,d} = \sqrt{\frac{8n}{h(d)}} + 2$. So, any $G \in \G_d$ on $n$ vertices is well-burnable if 
    \begin{equation}\label{eqn:whenburnable}
    \sqrt{\frac{8n}{h(d)}} + 2 \leq \sqrt{n}.
    \end{equation} 
    Now since $h(d)>8$, we have $\sqrt{\frac{8}{h(d)}} < 1$, so equation~(\ref{eqn:whenburnable}) holds if and only if \[\sqrt{n} \geq \frac{2}{1 - \sqrt{\frac{8}{h(d)}}}\] and so the result follows.
\end{proof}
Theorem~\ref{thm:main2} tells us that if we have a class of connected graphs in which $r$-neighborhoods have size at least linear in $r$ (for $r \geq 2$), with high enough slope of linearity, then graphs in the class are well-burnable if they have sufficiently many vertices. This is despite any condition on degree lower bounds for individual vertices. As a concrete example, even if we can't get a universal lower bound on vertex degrees for a connected graph $G$, if we know that $|N_r(v)| \geq 9r$ for any $r \in \{2,3,\ldots,\rad(G)\}$ and any vertex $v$, then $G$ is well-burnable provided that $G$ has slightly more than $1200$ vertices. Theorem~\ref{thm:main2} is just an example of the versatility of Theorem~\ref{thm:main}.

Since Theorem~\ref{thm:trianglefree} and Theorem~\ref{thm:main2} follow as a result of Theorem~\ref{thm:main}, we are motivated to prove Theorem~\ref{thm:main}. The proof can be seen as a generalization of the proof of Theorem 1 in \cite{kamali2020burning}.

\begin{proof}(of Theorem~\ref{thm:main})
Fix a positive integer $r \in \{1,2,\ldots,\rad(G)\}$ and let $A=\{v_1,v_2,\ldots,v_{|A|}\}$ be a maximal set of vertices in $G$ with respect to having disjoint closed $r$-neighborhoods, so the sets $N_r(v_1),\ldots,N_r(v_{|A|})$ are disjoint. By Lemma~\ref{lem:maximal}, $b(G) \leq |A|+2r$. Since $G \in \G_d$, by our assumptions, we have $|N_r(v_i)| \geq f_d(r)$ for any $i$. From this,
\[
n \geq  \sum_{i=1}^{|A|} |N_r(v_i)| \geq  |A| \cdot f_d(r),
\]
so
\begin{equation}\label{eqn:burninginequality2}
b(G) \leq |A|+2r \leq \frac{n}{f_d(r)}+2r=g_{n,d}(r).    
\end{equation}
Now suppose $\rad(G) \leq m_{n,d}$. Select $v$ so that $\ecc(v)=\rad(G)$ and start the burning process with $v$. This will burn the graph in at most $\rad(G) \leq m_{n,d}$ rounds.

Otherwise, $\rad(G) > m_{n,d}$. Since inequality~(\ref{eqn:burninginequality2}) is true for each $r \in \{1,2,\ldots,\rad(G)\}$, we have 
\begin{equation}\label{eqn:inequality}
b(G) \leq  \min_{r \in \{1,2,\ldots,\rad(G)\}} g_{n,d}(r). \\
\end{equation}
Since $1 \leq m_{n,d} \leq \rad(G)$, both $\lfloor m_{n,d} \rfloor$ and $\lceil m_{n,d} \rceil$ are in $\{1,2,\ldots,\rad(G)\}$ so by inequality (\ref{eqn:inequality}), $b(G) \leq g_{n,d} \left( \lfloor m_{n,d} \rfloor\right)$ and $b(G) \leq g_{n,d} \left( \lceil m_{n,d} \rceil\right)$. Merging the cases when $1 \leq m_{n,d} \leq \rad(G)$ and $\rad(G)<m_{n,d}$, we conclude that $b(G)$ is bounded above by any upper bound on $\lceil m_{n,d} \rceil$ and $\min \left \{g_{n,d} \left( \lfloor m_{n,d} \rfloor \right), g_{n,d} \left( \lceil m_{n,d} \rceil \right)\right \}$.
\end{proof}
It should be noted that Theorem~\ref{thm:main} illustrates an issue with the result in Theorem 1 of \cite{kamali2020burning}. Therein, the authors use an instance of the framework here with $m_{n,d}=\sqrt{\frac{3n}{2(d+1)}}$ and evaluate their analogous function $g_{n,d}(x)$ at this value. However, this value $m_{n,d}$ is likely not an integer and it may be the case that $g_{n,d}(m_{n,d})$ is actually less than $\min_{r \in \{1,2,\ldots,\rad(G)\}} g_{n,d}(r)$. One can check that using the framework of Theorem~\ref{thm:main}  their same theorem can be proven with $23$ altered to $36$. Regardless, their work still establishes the significant breakthrough that connected graphs whose vertices have sufficiently high degree are well-burnable.

\section{Well-Burnable Trees \& Degree Statistics}\label{sec:trees}

By the Tree Reduction Theorem \cite{bonato2016burn} (see Theorem~\ref{thm:treereduction}) it suffices to prove the Burning Number Conjecture for trees. With this as motivation, this section establishes an inequality on degree statistics of trees that, if satisfied, guarantees a tree is well-burnable. This main result is Theorem~\ref{thm:trees}, and we show through Theorem~\ref{thm:trees2}, Theorem~\ref{thm:prob} and Theorem~\ref{thm:prob2} how it serves as a paradigm for establishing many classes of trees are well-burnable. As is the theme of this article, we hope to see many more applications of the general paradigm afforded by Theorem~\ref{thm:trees}.

In \cite{burningsurvey}, Bonato makes reference to Theorem~\ref{thm:23} and continues to say ``Although this result encompasses a large class of graphs, it omits the class of trees''.  He further states ``it would be interesting to consider other classes where [the conjecture] holds, such as prescribed classes of trees''. We address these statements in Theorem~\ref{thm:trees2} by proving any \emph{tree} whose non-leaf vertices have degree at least $4$ is well-burnable, and any tree on at least $81$ vertices whose non-leaf vertices have degree at least $3$ is well-burnable. This result complements the discoveries in Theorem~\ref{thm:improved1} and Theorem~\ref{thm:improved2}.

The big hole that remains is what to do in the presence of vertices of degree $2$; as soon as even one such vertex appears, all the previous results become irrelevant. We address this in Theorem~\ref{thm:prob} and Theorem~\ref{thm:prob2}, both consequences of Theorem~\ref{thm:trees}, by proving trees that do not have a high concentration of degree $2$ vertices among their non-leaf vertices, or have a low concentration of high degree vertices therein, are forced to be well-burnable (provided they have sufficiently many non-leaf vertices).

Throughout this section, if $G$ is a tree on $n$ vertices, we write $n'$ for the number of non-leaf vertices in $G$. Furthermore, we write $n_k$ for the number of vertices of degree $k$.
%

\begin{theorem}\label{thm:trees}
If $G$ is a tree and
\[
\left \lceil 2 \sqrt{\frac{n_2+n_3+n_4+\cdots}{3}} \ \right \rceil +2 \leq \lceil \sqrt{2+n_2+2n_3+3n_4+\cdots} \ \rceil.
\]
then $G$ is well-burnable.
\end{theorem}
\begin{proof}
Let $G'$ be the graph obtained from $G$ by deleting its degree $1$ vertices. Then $G'$ has $n'$ many vertices. The graph $G'$ is itself connected so by \cite{improvedburning} we can burn $G'$ in at most $\left \lceil \sqrt{\frac{4n'}{3}} \right \rceil +1$ rounds.  We can subsequently burn $G$ in at most $\left \lceil \sqrt{\frac{4n'}{3}} \right \rceil +2$ rounds by burning the vertices of $G$ in $G'$ and needing at most one more round to burn the degree $1$ vertices of $G$. As a consequence, $G$ is well-burnable if $\left \lceil \sqrt{\frac{4n'}{3}} \right \rceil +2 \leq \lceil \sqrt{n} \ \rceil$. Now by the Handshake Lemma,
\begin{align*}
2(n-1) &= n_1+2n_2+3n_3+\cdots \\
&= (n_1+n_2+n_3+\cdots)+(n_2+2n_3+3n_4+\cdots) \\
&= n + n_2+2n_3+\cdots,\\
\end{align*}
so we have $n=2+n_2+2n_3+3n_4+\cdots$ whereas $n'=n_2+n_3+n_4+\cdots$ and this translates to $G$ being well-burnable if
\[
\left \lceil 2 \sqrt{\frac{n_2+n_3+n_4+\cdots}{3}} \ \right \rceil +2 \leq \lceil \sqrt{2+n_2+2n_3+3n_4+\cdots} \ \rceil.
\]
\end{proof}
We can now address trees whose non-leaf vertices have degree lower bounds.
\begin{theorem}\label{thm:trees2}
Let $\mathcal{T}_d$ be the class of trees whose non-leaf vertices have degree at least $d$. Then every graph in $\mathcal{T}_d$ is well-burnable for $d \geq 4$, and every graph in $\mathcal{T}_3$ on at least $81$ vertices is well-burnable.
\end{theorem}
\begin{proof}
Suppose every non-leaf vertex of $G$ has degree at least $d$. Then as in the proof of Theorem~\ref{thm:trees},
\begin{align*}
n &=2+(d-1)n_d+dn_{d+1}+\cdots \\
&\geq (d-1)(n_d+n_{d+1}+\cdots) = (d-1)n'
\end{align*}  
and so $n' \leq \frac{n-2}{d-1}$. Since $n'$ is integer, this implies $n' \leq \left \lfloor \frac{n-2}{d-1} \right \rfloor$. Subsequently by Theorem~\ref{thm:trees}, $G$ is well-burnable provided that
\begin{equation}\label{floorceiling}
\left \lceil 2\sqrt{\frac{1}{3} \left \lfloor \frac{n-2}{d-1} \right \rfloor} \ \right \rceil +2 \leq \lceil \sqrt{n} \rceil.
\end{equation}
First consider when $d \geq 4$. For such $d$, inequality~(\ref{floorceiling}) is satisfied for $n \geq 25$, so it remains to resolve the trees in $\mathcal{T}_d$ with $n \leq 24$ vertices. In any such tree $G$, $n' \leq \lfloor \frac{n-2}{3} \rfloor  \leq \lfloor \frac{22}{3} \rfloor = 7$, so if we delete the degree $1$ vertices we are left with a tree $G'$ on at most $7$ vertices. It can be verified that trees on at most $7$ vertices are well-burnable and hence $G'$ can be burned in at most $\lceil \sqrt{7} \rceil = 3$ rounds, so $G$ can be burned in at most $4$ rounds. From this, if $n \geq 10$, $b(G) \leq 4 = \lceil \sqrt{10} \rceil \leq \lceil \sqrt{n} \rceil$. It remains to resolve the trees in $\mathcal{T}_d$ with $n \leq 9$ vertices. There are very few graphs in $\mathcal{T}_d$ on $n$ vertices where $d \geq 4$ and $n \leq 9$, and all these can be verified to be well-burnable by direct computation. Altogether, every graph in $\mathcal{T}_d$ is well-burnable if $d \geq 4$. Finally for $d=3$ one can check that inequality~(\ref{floorceiling}) is satisfied for $n \geq 81$.
\end{proof}

We now suggest a method for accounting for the presence of degree $2$ vertices. If we look at the inequality in Theorem~\ref{thm:trees} we might suspect that if the concentration of degree $2$ vertices among the non-leaf vertices is not too high, then our given graph is well-burnable. Let's quantify this. Let $p \in [0,1]$ and suppose that $pn'$ of the non-leaf vertices have degree $2$, and so $(1-p)n'$ have degree at least $3$. Since $n = 2+n_2 + 2n_3 + 3n_4 + \cdots \geq 2+pn'+2(1-p)n'$, Theorem~\ref{thm:trees} tells us a tree $G$ is well-burnable provided that
\[
\left \lceil 2 \sqrt{\frac{n'}{3}} \ \right \rceil +2 \leq \lceil \sqrt{2+(p+2(1-p))n'} \ \rceil
\]
and this occurs for sufficiently large $n'$ provided that $\frac{2}{\sqrt{3}} < \sqrt{p+2(1-p)}$ or equivalently $p<\frac{2}{3}$. In conclusion, if fewer than $\frac{2}{3}$ of the non-leaf vertices in a tree have degree $2$, and the tree has sufficiently many non-leaf vertices, it is well-burnable.
\begin{theorem}\label{thm:prob}
Let $p \in \left[0,\frac{2}{3}\right)$, and let $\mathcal{G}_p$ be the class of trees whose concentration of degree $2$ vertices among all non-leaf vertices is $p$, i.e. $p=\frac{n_2}{n'}$. Then there is a constant $N_p>0$ so that any graph in $\mathcal{G}_p$ on at least $N_p$ non-leaf vertices is well-burnable.
\end{theorem}
Theorem~\ref{thm:prob} can be made stronger as the graph burning literature introduces better upper bounds for $b(G)$. Suppose for instance that one proves any graph $G$ on $n$ vertices has $b(G) \leq \lceil C \sqrt{n} \rceil + O(1)$ for some constant $C$ that is very close to but greater than $1$. Applying the same developments as in the ones leading to Theorem~\ref{thm:prob}, we will have the same result but with $p$ ranging in $\left[0,2-C^2\right)$. As $C>1$ but close to $1$, this allows $p$ to be less than but very close to $1$ itself, so we can have increasingly higher concentrations of degree $2$ vertices among the non-leaves and still be well-burnable. 

Finally, we can use Theorem~\ref{thm:trees} to get an even better sense of how degree distribution forces trees to be well-burnable. For any $k \geq 2$, let $p_k = \frac{n_k}{n'}$, the concentration of degree $k$ vertices among the non-leaf vertices. We see by eliminating ceilings and dividing by $\sqrt{n'}$ in Theorem~\ref{thm:trees} that $G$ is well-burnable provided that
\[
\frac{2}{\sqrt{3}} \cdot \sqrt{p_2+p_3+p_4+\cdots}  + \frac{3}{\sqrt{n'}} \leq \sqrt{\frac{2}{n'} + p_2 + 2p_3 + 3p_4 + \cdots}.
\]
If we stabilize the concentrations $\{p_k\}_{k \geq 2}$ and let the number of non-leaf vertices grow then this inequality will be satisfied when 
\begin{equation}\label{eqn:prob}
\frac{2}{\sqrt{3}} \cdot \sqrt{p_2+p_3+p_4+\cdots}  < \sqrt{p_2 + 2p_3 + 3p_4 + \cdots} 
\end{equation}
for sufficiently large $n'$. This gives rise to the following theorem:
\begin{theorem}\label{thm:prob2}
Let $G$ be a tree on $n'$ non-leaf vertices, and let $p_k$ be the concentration of degree $k$ vertices among the non-leaf vertices. For sufficiently large $n'$, if
\[
p_4 + 2p_5 + 3p_6 + \cdots \ > \ \frac{1}{3}
\]
then $G$ is well-burnable. In particular, this holds if for some $k \geq 4$ the concentration of vertices of degree at \emph{least} $k$ among the non-leaf vertices exceeds $\frac{1}{3(k-3)}$.
\end{theorem}
\begin{proof}
By Theorem~\ref{thm:prob} it is sufficient to show that $p_2<\frac{2}{3}$. For the first part, this follows from the given inequality and squaring inequality~(\ref{eqn:prob}). For the second part, observe our given inequality is satisfied when $(k-3)p_k+(k-2)p_{k+1}+ \cdots > \frac{1}{3}$ and this in turn is satisfied when the stronger inequality $(k-3)(p_k+p_{k+1}+\cdots)>\frac{1}{3}$ is satisfied. The result follows.
\end{proof}
An illustration of Theorem~\ref{thm:prob2} shows how powerful it can be: if a tree has enough non-leaf vertices and 10\% of them have degree at least $7$, then $G$ is well-burnable, even if the overwhelming majority of the remaining 90\% of the non-leaf vertices have degree $2$.

\section{Future Directions}

The developments in Section~\ref{sec:families} hinge on Lemma~\ref{lem:maximal} which provides an upper bound on the burning number based on disjoint and uniform depth neighborhoods. Perhaps the most direct opportunity for improving the bounds therein is in adapting to neighborhoods of varying sizes. The difficulty with adapting the lemma in this manner is that maximal disjoint neighborhoods of uniform depth allowed us to control the number of additional vertices needed to end the burning process. If we instead, say, picked maximal disjoint neighborhoods whose depths could be one of two possible numbers, controlling the additional number of steps needed to end the burning process could depend heavily on the graph. However, it might be possible to at least achieve some sufficient bounds. A possible insight is to study the $A$-burnable constructions developed by Land and Lu \cite{land2016upper} to see what might be amenable.

Our constructions leave a lot more to be said about trees as well. It would be fruitful to discover a complete characterization of the degree sequences of trees that satisfy the inequality in Theorem~\ref{thm:trees}. The inequality is challenging to work with because of the integer rounding effects, however it is possible to relax these inequalities, make some assumptions about vertex degrees, and still develop substantial insight beyond what was garnered in Section~\ref{sec:trees}. Finally, it would be nice to completely resolve Theorem~\ref{thm:trees2} for $\mathcal{T}_3$. By the arguments in Theorem~\ref{thm:trees2}, this is a computational feat that amounts to checking all trees in $\mathcal{T}_d$ on at most $80$ vertices. It might be particularly helpful to note, again as in the proof of Theorem~\ref{thm:trees2}, that any such tree has at most $39$ non-leaf vertices.

\bibliographystyle{plain}
\bibliography{references}

\end{document}